\documentclass[12pt,oneside, italian]{article}

\usepackage[english]{babel}

\usepackage{cancel}

\usepackage{cite}

\usepackage{latexsym}

\usepackage{amssymb,amsthm,amsmath}

\usepackage{graphicx,color}

\title{\bf Qualche Idea}

\author{}

\newtheorem{theorem}{Theorem}[section]

\newtheorem{proposition}[theorem]{Proposition}

\newtheorem{lemma}[theorem]{Lemma}

\newtheorem{remark}[theorem]{Remark}

\newtheorem{example}[theorem]{Example}

{\theoremstyle{definition}

\newtheorem*{definition*}{Definition}

}

\newtheorem*{proposition*}{Proposition}

\newtheorem*{corollary*}{Corollary}

\newtheorem*{lemma*}{Lemma}

\newtheorem*{remark*}{Remark}

\newcommand{\cX}{\mathcal X}
\newcommand{\cL}{\mathcal L}

\newcommand{\fq}{\mathbb {F}_q}
\newcommand{\N}{\mathbb {N}}
\newcommand{\Z}{\mathbb {Z}}
\newcommand{\bP}{\mathbb {P}}
\def\cD{\mathcal D}

\title{Algebraic Geometric Codes on Many Points from Kummer Extensions}
\author{D. Bartoli, L. Quoos, and G. Zini}

\date{}

\begin{document}

\maketitle 

\begin{abstract}
For Kummer
extensions defined by $y^m = f (x)$, where $f (x)$ is a separable polynomial over the finite field $\mathbb{F}_q$,
we compute the number of Weierstrass gaps at two totally ramified places. For many totally ramified places we give a criterion to find pure gaps at these points and present families of pure gaps. We then apply our
results to construct $n$ -points algebraic geometric codes with good parameters.
\end{abstract}

{\bf Keywords:} Weierstrass semigroups, algebraic geometric codes, codes on many points, Kummer extensions.

{\bf MSC:} 11G20, 14G50, 14H55.

\footnote{

Daniele Bartoli is with the Dipartimento di Matematica e Informatica, Universit\`a degli Studi di Perugia, Via Vanvitelli 1 - 06123 Perugia - Italy, {\em email:} {daniele.bartoli}.

Giovanni Zini is with the Dipartimento di Matematica e Informatica ``Ulisse Dini'', Universit\`a degli Studi di Firenze, Viale Morgagni 67/A - 50134 Firenze - Italy, {\em email: }{gzini@math.unifi.it}.

Luciane Quoos is with the Instituto de Matem\'atica, Universidade Federal do Rio de Janeiro, Rio de Janeiro 21941-909 - Brazil, {\em email: } {luciane@im.ufrj.br}.
}

\section{Introduction}
In the early eighties tools from algebraic geometry were applied by V. Goppa  to construct linear codes using algebraic curves over finite fields, see \cite{G1982}. Nowadays these codes are called algebraic-geometric codes, AG codes for short. The starting point in the construction of an AG code is a projective, absolutely irreducible, non singular algebraic curve $\cX$ of genus $g \geq 1$ defined over the finite field $\fq$ with cardinality $q$. Let $F=\fq(\cX)$ be its function field with $\fq$ being the field of constants. Consider $Q_1, \dots , Q_n$ pairwise distinct rational places on $F$. Let $D=Q_1+ \dots +Q_n$ and $G$ be divisors such that $Q_i$ is not in the support of $G$ for $i=1, \dots, n$. The linear code $C_{\Omega}(D,G)$ is defined by
$$
C_{\Omega}(D,G)=\{(\mbox{res}_{Q_1}(\eta),\ldots,\mbox{res}_{Q_n}(\eta))\mid \eta\in \Omega(G-D)\}\subseteq \fq^n,
$$
where $\Omega(G-D)$ is the space of $\fq$-rational differentials $\eta$ on $\mathcal{X}$ such that either $\eta=0$ or $\mbox{div}(\eta)\succeq G-D$ and $\mbox{res}_{Q_j}\eta$ is the residue of $\eta$ at $Q_j$.

The code $C_{\Omega}(D,G)$ has length $n$ and dimension $k = i(G-D) - i(G)$
where $i(G)$ denotes the speciality index of the divisor $G$.
We say that $C_\Omega(D,G)$ is an $[n, k, d]$-code where $d$ denotes the minimum distance of the code. One of the main features of this code is that its minimum distance $d$ satisfies the classical Goppa bound,
namely
$$d \geq \deg G - (2g-2).$$
The integer $d^*=\deg G - (2g-2)$ is usually called the \emph{designed minimum distance}.
One way to obtain codes with good parameters is to find codes that improve the designed minimum distance.

If $G=\alpha P$ for some rational place $P$ on $F$ and $D$ is the sum of other
rational places on $\cX $, then the code $C_{\Omega}(D,G)$ is called an {\it one-point
AG code}. Analogously, if $G= \alpha_1 P_1+ \cdots +\alpha_nP_n$ for $n$ distinct rational
places $P_1,\ldots,P_n$ on $\cX$, then  $C_{\Omega}(D,G)$ is called a {\it $n$-point AG code}. For a more detailed introduction to AG codes, see \cite{HLP1998,S2009}. 

For a one-point divisor $G=\alpha P$ on the function field $F$, Garcia, Kim, and Lax \cite{GL1992,GKL1993} improved the designed minimum distance using the arithmetical structure of the Weierstrass semigroup at the rational place $P$. For a two-point divisor $G = \alpha_1P_1 + \alpha_2P_2$, Homma and Kim \cite{HK2001} introduced the notion of pure gaps and obtained similar results. By choosing $\alpha_1$ and $\alpha_2$ satisfying certain arithmetical conditions depending on the structure of the Weierstrass semigroup at $P_1$ and $P_2$, they improved the designed minimum distance. Matthews \cite{G2001}  showed that for an arbitrary curve there exist two-point AG codes that have better parameters than any comparable one-point AG code constructed from the same curve.
Finally, for divisors $G= \alpha_1 P_1+ \cdots +\alpha_nP_n$ at $n$ distinct rational places on $\cX$, results from the theory of generalized Weierstrass semigroups and pure gaps were obtained by Carvalho and Torres \cite{CT2005}. They have been used to obtain AG codes whose minimum distance beats the classical Goppa bound on the minimum distance, see Theorem \ref{distmanypoints}.
 
Many applications and results on AG codes can be found for one- and two-point codes in \cite{HK2001, DK2011, ST2014, MQS2016}, and  for $n$-point codes in \cite{G2004, CT2005, CK2007}. 
The minimum distances of several AG codes have been studied in the case when $\cX$ is a Kummer curve. For instance, when $\cX$ is the Hermitian curve results can be found in \cite{HK2001, G2001, HK2006}, or a subcover of the Hermitian curve in \cite{M2004}, or a generalization of the Hermitian curve in \cite{ST2014}.

In this paper we analyze $n$-point codes when $F$ is a Kummer extension defined by $y^m= f(x),$ where $ f(x) \in \fq[x]$ is a separable polynomial of degree $r$ coprime to $m$. We extend results by Castellanos, Masuda, and Quoos \cite{MQS2016} on the Weierstrass semigroup at two rational places $P_1$ and $P_2$. In particular, for a class of Kummer curves we explicitly compute  the number of gaps at $P_1, P_2$, see Theorem \ref{twopoints}, generalizing a result by Matthews {\rm \cite[Theorem 3.6]{G2001}}.
For Kummer extensions, we also study the Weierstrass semigroup at many rational places under some hypothesis on the places. We give an arithmetic characterization of pure gaps (Propositions \ref{puregapsmanypoints} and \ref{puregapsinfty}) and apply it to a large family of Kummer extensions to provide families of pure gaps (Propositions \ref{puregapstwopoints}, \ref{manypoints1}, and \ref{manypoints2}).
We obtain codes such that the Singleton defect $\delta=n+1-k-d$ is improved, see Remarks \ref{applications1} and \ref{applications2}. We illustrate our results constructing AG codes on many points from the the Hermitian function field, and observe that the best improvements on the minimum distance with respect to the corresponding ones in the MinT's Tables~\cite{MinT} are obtained by two- or three-point codes.

The paper is organized as follows.
In Section \ref{Sec:Preliminary} we set the notations and present the preliminary results on the Weierstrass semigroup at two and many points. In Sections \ref{Sec:TwoPoints} and \ref{Sec:ManyPoints} we consider a large class of Kummer curves.
In particular, in Section \ref{Sec:TwoPoints} we study the Weierstrass semigroup at two totally ramified rational points and compute the number of gaps at them. In Section \ref{Sec:ManyPoints} we give an arithmetic characterization of pure gaps at many points which provides families of pure gaps. We apply them to the construction of AG codes improving the Singleton defect. We illustrate our results constructing AG codes on many points from the Hermitian curve, see Example \ref{ExHerm}.

\section{Preliminary results}\label{Sec:Preliminary}
Let $\cX$ be a projective, absolutely irreducible, nonsingular algebraic curve of genus $g$ defined over the finite field $\fq$. Let $F=\fq(\cX)$ be its function field with the field of constants $\fq$. For a function $z$ in $F$, $(z)$ and $(z)_\infty$ stand for its principal and polar divisor, respectively. We denote by $\bP(F)$ the set of places of $F$ and by $\cD_F$ the free abelian group generated by the places of $F$. The elements $D$ of $\cD_F$ are called {\it divisors} and can be written as
$$
D=\sum_{P\in \bP(F)}n_P\,P\quad \text{ with } n_P\in \Z,\;  n_P=0\text{ for almost all }P\in\bP(F).
$$
The degree of a divisor $D$ is $\deg(D)=\sum\limits_{P\in \bP(F)}n_P \cdot\deg P$, where $\deg P$ is the degree of the place $P$ over $\fq$.
Given a divisor $D\in\cD_F$, the Riemann-Roch vector space associated to $D$ is defined by
$
\cL(D):=\{z\in F\,|\, (z)\ge -D\}\cup \{0\}.
$
We denote by $\ell(D)$ the dimension of $\cL(D)$ as a vector space over the field of constants $\fq$. From the Riemann-Roch Theorem, it follows that, for divisors $D$ such that $2g-1 < \deg D$, we have $\ell(D) = \deg(D)+1-g$, see \cite[Th. 1.5.17]{S2009}. 

Let $\N$ be the set of non-negative integers.
For distinct rational places $P_1, \dots, P_s$ on $\bP(F)$, let
$$
H(P_1, \dots, P_s)= \{(n_1, \dots , n_s) \in \N^s \ | \ \exists \ z \in F \text{ with } (z)_\infty = n_1P_1 + \cdots + n_sP_s \}
$$
be the Weierstrass semigroup at $P_1, \dots , P_s$.  The complement $G(P_1, \dots, P_s)=\N^s \setminus H(P_1, \dots, P_s)$ is always a finite set and its elements are called {\it Weierstrass gaps} at $P_1, \dots, P_s$.
A gap can be characterized in terms of the dimension of certain Riemann-Roch spaces, more specifically, an $s$-tuple $(n_1, \dots , n_s) \in \N^s$ is a gap at $P_1, \dots, P_s$ if and only if $\ell\big(\sum_{i=1}^s n_iP_i\big)= \ell\big((\sum_{i=1}^s n_iP_i)- P_j\big)$ for some $j \in \{1, \dots, s\}$.

For $s=1$, the semigroup $H(P_1)$ is the well-known Weierstrass semigroup at one point on the curve and $G(P_1)$ has exactly $g$ gaps. For $s \geq 2$, the number of gaps may vary depending on the choice of the points.
When $s=2$, the size of $G(P_1,P_2)$ was given by M. Homma \cite{H1996} in terms of $G(P_1)$ and $G(P_2)$ as follows. Let $1=a_1<a_2<\cdots<a_g$ and $1=b_1<b_2<\cdots<b_g$ be the gap sequences at $P_1$ and $P_2$, respectively. For $i=1,\ldots, g$, let $\gamma(a_i) = \min \{ b \in G(P_2) \mid  (a_i, b) \in H(P_{1}, P_{2}) \}$. By \cite[Lemma 2.6]{K1994}, $\{ \gamma(a_i) \mid i=1, \dots, g \} = G(P_{2})$. Therefore, there exists a permutation $\sigma$ of the set $\{1, \ldots , g\}$ such that $\gamma(a_i)=b_{\sigma(i)}$, and $$\Gamma(P_{1}, P_{2})=\{ (a_{i} , b_{\sigma(i)}) \mid  i=1, \ldots ,g\}$$ is the graph of a bijective map $\gamma $ between $G(P_{1})$ and $G(P_{2})$.
Define $$r(P_1,P_2)=|\{(x,y)\in\Gamma(P_1,P_2)\mid x<y,\gamma(x)>\gamma(y)\}|$$
the number of inversions for $\gamma$.

\begin{theorem}[{\!\!\cite[Theorem 1]{H1996}}]\label{numberofgaps}
Under the above notation, the number of gaps at $P_1,P_2$ is
$$ |G(P_1,P_2)|=\sum_{i=1}^g a_i + \sum_{i=1}^g b_i - r(P_1,P_2). $$
\end{theorem}
A characterization of $\Gamma(P_1,P_2)$ is the following.

\begin{lemma}[{\!\!\cite[Lemma 2]{H1996}}] \label{lemmagamma}
Let $\Gamma '$ be a subset of $(G(P_{1}) \times G(P_{2})) \cap H(P_{1},P_{2})$. If there exists a permutation $\tau$ of $\{ 1, \ldots , g\}$ such that $\Gamma ' = \{ (a_{i} , b_{\tau(i)}) \mid  i=1, \ldots ,g \}$, then $\Gamma ' = \Gamma(P_{1}, P_{2})$.
\end{lemma}

The Weierstrass semigroup $H(P_{1}, P_{2})$ can be recovered from $\Gamma (P_{1}, P_{2})$ as follows. For $\mathbf{x} = (a_1, b_1),\mathbf{y} = (a_2, b_2) \in \mathbb{N}^2$, define the \textit{least upper bound} of $\mathbf{x}$ and $\mathbf{y}$ as $\mathrm{lub}(\mathbf{x},\mathbf{y})= \left(\max\{a_1, a_2\}, \max\{b_1, b_2\}\right)$.
Then, by \cite[Lemma 2.2]{K1994},
\begin{equation}\label{lub}
H(P_{1},P_{2}) = \{ \mathrm{lub} (\mathbf{x},\mathbf{y}) \mid \mathbf{x},\mathbf{y} \in \Gamma(P_{1}, P_{2}) \cup (H(P_{1}) \times \{0\}) \cup (\{0\} \times H(P_{2})) \}.
\end{equation}

We now introduce the important concept of pure gaps that will be used in the construction of AG codes. An $s-$tuple $(n_1, \dots, n_s)\in\N^s$ is a {\it pure gap} at $P_1, \dots , P_s$ if 
$$\ell\Big(\sum_{i=1}^s n_iP_i\Big)= \ell\Big(\big(\sum_{i=1}^s n_iP_i\big)- P_j\Big)  \text{ for all } j=1, \dots, s.$$ 
The set of pure gaps at $P_1,\ldots,P_s$ is denoted by $G_0(P_1,\ldots,P_s)$. Clearly, a pure gap is always a gap.

\begin{lemma}[{\!\!\cite[Lemma 2.5]{CT2005}}]\label{purecondition}
An $s$-tuple $(n_1, \dots, n_s)$ is a pure gap at $P_1, \dots ,P_s$ if and only if $\ell\big(\sum_{i=1}^s n_iP_i\big)= \ell\big(\sum_{i=1}^s (n_i-1)P_i\big)$.
\end{lemma}

Pure gaps can be used to improve the designed minimum distance of AG codes.

\begin{theorem}[{\!\!\cite[Theorem 3.4]{CT2005}}] \label{distmanypoints}
Let $P_1, \dots , P_s,Q_1, \dots, Q_n$ be pairwise distinct $\fq$-rational points on $\cX$ and $(a_1, \dots , a_s),(b_1, \dots , b_s)\in\N^s$ be two pure gaps at $P_1, \dots, P_s$. Consider the divisors $D= Q_1+ \cdots + Q_n$ and $G= \sum_{i=1}^s (a_i+b_i-1)P_i$.
Suppose that $a_i \leq b_i$ for all $i=1, \dots , s$, and that each $s$-tuple $(c_1, \dots , c_s)\in\N^s$ with $a_i \leq c_i \leq b_i$ for $i=1,\ldots,s$, is also a pure gap at $P_1, \dots , P_s$. The the minimum distance $d$ of $C_{\Omega}(D,G)$ satisfies 
$$ d \geq \deg(G) -(2g-2)+s+\sum_{i=1}^s (b_i-a_i).$$
\end{theorem}

Hereafter we work on a Kummer extension $F=\fq(x,y)/\fq(x)$ defined by 
$y^m=f(x), \, m\geq2, \, p\nmid m, f(x)$ a separable polynomial of degree $r$ in  $\mathbb{F}_q[x]$
and $\gcd(m,r)=1$. We denote by
$P_1,\ldots ,P_s, P_\infty (s \leq r),$ the rational places of $F$ which are totally ramified in the extension $F/ \fq(x)$, and $P_\infty$ is the pole of $x$. The genus $g$ of $F$ is $(m-1)(r-1)/2$.

We use a result by Maharaj \cite{M2004} to build up an arithmetic characterization of pure gaps at many points in a \emph{Kummer extension}. Firstly we need the definition of the restriction of a divisor in a function field extension $F/K$. For any divisor $D$ of $F$ and any intermediate field $K \subseteq E\subseteq F$, write
$D = \sum_{R\in \bP(E)}\;\sum_{Q\in \bP(F),\,Q|R}\, n_Q\, Q$.
We define the restriction of $D$ to $E$ as
$$
D\Big|_{E}= \sum\limits_{R\in \bP(E)} \min\,\left\{\left\lfloor\frac{n_Q}{e(Q|R)}\right\rfloor\colon
{Q|R}\right\}\,R,
$$
where $e(Q|R)$ is the ramification index of $Q$ over $R$.

\begin{theorem}[{\!\!\cite[Theorem 2.2]{M2004}}] \label{ThMaharaj}
Let $F/\fq(x)$ be a Kummer extension of degree $m$ defined by $y^m=f(x)$. Then, for any divisor $D$ of $F$ that is invariant under the action of $Gal(F/\fq(x))$, we have that
$$ \mathcal{L}(D)= \bigoplus\limits_{t=0}^{m-1} \mathcal L\left(\left[D+(y^t)\right]\Big|_{\fq(x)}\right)\,y^t,$$
where $\left[D+(y^t)\right]\Big|_{\fq(x)}$ denotes the restriction of the divisor $D+(y^t)$ to $\fq(x)$.
\end{theorem}

\section{The Weierstrass semigroup at two points}\label{Sec:TwoPoints}
Let $F/\fq(x)$ be a Kummer extension defined by $y^m=f(x)$, where $f(x)\in\fq[x]$ is separable of degree $r$ coprime with $m$, and 
consider the Weierstrass semigroup $H(P,Q)$ at two rational places of $F$ which are totally ramified in $F/\fq(x)$.
As pointed out in Equation \eqref{lub}, the semigroup $H(P,Q)$ is related to the set $\Gamma(P,Q)$, and \cite[Theorem 4.3]{MQS2016} yields 
\begin{equation*}\label{gammainfty}
\Gamma(P_\infty, P_1)= \left\{ \left(mr-mj-ri, i+m(j-1)\right) \mid  1 \leq i \leq m-1-\left\lfloor \frac{m}{r}\right\rfloor, 1\leq j \leq r-1- \left\lfloor \frac{ri}{m}\right\rfloor \right\},
\end{equation*}
where $P_\infty$ is the unique pole of $x$ and $P_1$ is another totally ramified place.
We now compute $\Gamma(P_1,P_2)$, where $P_1$ and $P_2$ are two distinct rational places of $F$ different from $P_\infty$ and totally ramified in the extension $F/\fq(x)$.
\begin{proposition}\label{gammaP1P2}
Let $F/\fq(x)$ be a Kummer extension defined by $y^m=f(x)$, where $f(x)\in\fq[x]$ is separable of degree $r$ and $\gcd(r,m)=1$. If $P_1$ and $P_2$ are two distinct totally ramified places of $F$ different from  $P_\infty$, then
$$
\Gamma(P_1,P_2)=\left\{\left(mi-j, m\left(\left\lceil\frac{rj}{m}\right\rceil-i\right)-j\right) \mid  1+\left\lfloor\frac{m}{r}\right\rfloor \leq j \leq m-1,  1\leq i \leq  \left\lceil \frac{rj}{m}\right\rceil -1\right\}.
$$
\end{proposition}
\begin{proof}
For $\iota\in\{1,2\}$ let $\alpha_\iota\in\fq$ be such that $P_\iota$ is the unique zero of $x-\alpha_\iota$ in $F$. Let $i,j$ be positive integers and $k =\left\lceil\frac{jr}{m}\right\rceil-i$, so that $(i+k)m\geq jr$. By \cite[Prop. 3.1]{MQS2016}, the pole divisor of $\frac{y^{j}}{(x-\alpha_1)^i(x-\alpha_2)^k}$ is $(mi-j)P_1+(mk-j)P_2$.
Also, for $j\in\left\{ 1+\left\lfloor\frac{m}{r}\right\rfloor,\ldots, m-1\right\}$ and $h\in\left\{ 1,\ldots, \left\lceil \frac{rk}{m}\right\rceil -1\right\}$, we have that $(mh-j)\in G(P_1)\cap G(P_2)$ by \cite[Th. 3.2]{MQS2016}.
Hence, the set
$$
\Gamma^{\prime} = \left\{\left(mi-j, m\left(\left\lceil\frac{rj}{m}\right\rceil-i\right)-j\right)\ \mid \ 1+\left\lfloor\frac{m}{r}\right\rfloor \leq j \leq m-1,  1\leq i \leq  \left\lceil \frac{rj}{m}\right\rceil -1\right\}
$$
is a subset of $G(P_1)\times G(P_2) \cap H(P_1,P_2)$.
The cardinality of $\Gamma^{\prime}$ is
$$
|\Gamma^{\prime}| = \sum_{k=1+ \left\lfloor\frac{m}{r} \right\rfloor }^{m-1} \left( \left\lceil \frac{rk}{m} \right\rceil -1 \right)= \left(\sum_{k=1+ \left\lfloor\frac{m}{r} \right\rfloor }^{m-1} \left\lceil \frac{rk}{m} \right\rceil\right) -\left(m- \left\lfloor \frac{m}{r} \right\rfloor -1\right)$$
$$ = \left( \sum_{k=0}^{m-1} \left\lceil \frac{rk}{m} \right\rceil \right) - \left\lfloor\frac{m}{r} \right\rfloor -\left(m- \left\lfloor \frac{m}{r} \right\rfloor -1\right) =- \sum_{k=0}^{m-1} \left\lfloor \frac{-rk}{m} \right\rfloor -m+1$$
$$ = -\left(m-1\right)\left(-r-1\right)/2 -m+1= \left(m-1\right)\left(r-1\right)/2 = g,
$$
using \cite[Page 94]{GKP}.
Therefore $\Gamma'=\Gamma(P_1,P_2)$ by Lemma \ref{lemmagamma}.
\end{proof}

From Proposition \ref{gammaP1P2} we are able to compute the number of gaps at two totally ramified places in the case $m\equiv1\pmod r$.

\begin{theorem}\label{twopoints}
Let $F/\fq(x)$ be a Kummer extension defined by $y^m=f(x)$, where $f(x)\in\fq[x]$ is separable of degree $r$ and $\gcd(r,m)=1$.
Let $P_\infty\in\bP(F)$ be the pole of $x$ and $P_1\ne P_2$ be two other totally ramified rational places in $F/\fq(x)$. If $m=ur+1$ for some integer $u$, then
\begin{align*}
& | G(P_1,P_2) | = \frac{ur(r-1)(3ur^2-5ur+4r+4u-2)}{12}\text{, and }\\
& | G(P_\infty,P_1) | =  \frac{ur(r-1)(3ur^2-3ur+2r+2)}{12}.
\end{align*}
\end{theorem}
\begin{proof}
By Proposition \ref{gammaP1P2}, 
$$
\Gamma(P_1,P_2)=\left\{\left(mi-j, m\left(\left\lceil\frac{rj}{m}\right\rceil-i\right)-j\right)\ \mid \ 1+u \leq j \leq m-1,  1\leq i \leq  \left\lceil \frac{rj}{m}\right\rceil -1\right\}.
$$
Setting $(i_0, j_0)\in\N^2$ with $1+u \leq j_0 \leq m-1$ and $1\leq i_0 \leq \left \lceil \frac{rj_0}{m}\right\rceil -1$; by Theorem \ref{numberofgaps}, we need to count the number $r_{i_0,j_0}$ of pairs $(i_1, j_1)\in\N^2$ such that
\begin{small}
\begin{equation}\label{conditions}
1+u \leq j_1 \leq ru,\, 1\leq i_1 \leq \left \lceil \frac{rj_1}{m}\right\rceil -1,\,
m\left(i_0-i_1\right) < j_0-j_1,\, 
m\left(\left\lceil \frac{rj_1}{m} \right\rceil  -  \left\lceil  \frac{rj_0}{m} \right\rceil  +i_0-i_1 \right) < j_1-j_0 .
\end{equation}
\end{small}
For $h\in\{0,1\}$ write $j_h=k_h u+t_h$ with $k_h \in \left\{ 1, \dots, r-1\right\}$ and $t_h \in \left\{ 1, \dots , u\right\}$.
Then $\left \lceil \frac{rj_h}{m} \right\rceil = k_h+1$.
We split $r_{i_0,j_0}$ in a number of cases:
\begin{itemize}
\item $j_1=j_0$. Then \eqref{conditions} implies $i_0+1 \leq i_1 \leq  k_1$.  
\item $j_1 > j_0$ and $k_1=k_0$. Then \eqref{conditions} implies $1 \leq t_0 \leq u-1$, $t_1 \geq t_0+1$, and $i_0+1 \leq i_1 \leq k_1$. 
\item $j_1 > j_0$ and $k_1>k_0$. Then \eqref{conditions} implies $i_0+k_1-k_0 \leq i_1 \leq k_1$.
\item $j_1 < j_0$ and $k_1 < k_0$. Then \eqref{conditions} implies $1\leq t_0,t_1\leq u$, $1\leq i_0\leq k_1$, and $i_0\leq i_1 \leq k_1$.
\item $j_1 < j_0$ and $k_1 = k_0$. Then \eqref{conditions} implies $2 \leq t_0 \leq u$, $t_1 \leq t_0-1$, and $i_0+1 \leq i_1 \leq k_1$.
\end{itemize}
By direct computation, this yields
\begin{align*}
r(P_1,P_2) & =\sum_{(i_0,j_0)\in\Gamma(P_1,P_2)} r_{i_0,j_0} =\sum_{k_0=1}^{r-1} \sum_{t_0=1}^u \sum_{i_0=1}^{k_0} (k_0-i_0)+
\sum_{k_0=1}^{r-1} \sum_{t_0=1}^{u-1} \sum_{t_1=t_0+1}^u\sum_{i_0=1}^{k_0} (k_0-i_0)\\
&+ \sum_{k_0=1}^{r-2} \sum_{t_0=1}^{u} \sum_{k_1=k_0+1}^{r-1}\sum_{t_1=1}^u \sum_{i_0=1}^{k_0} (k_0-i_0+1)+
\sum_{k_0=2}^{r-1} \sum_{t_0=1}^{u} \sum_{k_1=1}^{k_0-1}\sum_{t_1=1}^u\sum_{i_0=1}^{k_1} (k_1-i_0+1)\\
&+ \sum_{k_0=1}^{r-1} \sum_{t_0=2}^{u} \sum_{t_1=1}^{t_0-1}\sum_{i_0=1}^{k_0} (k_0-i_0)
 = \frac{u^2(r-2)(r-1)r(r+3)}{12}.
\end{align*}

Also, by \cite[Th. 3.2]{MQS2016}, we have
\begin{align}
\sum_{n \in G(P_1)} n & = \sum_{n \in G(P_2)} n = \sum_{j=1+u}^{m-1} \sum_{i=1}^{\left\lceil \frac{rj}{m} \right\rceil-1} \left(mi-j\right)
 =\sum_{k=1}^{r-1} \sum_{t=1}^{u} \sum_{i=1}^{k-1} \left((ur+1)i-(ku+t)\right) \\
&= \frac{ur(r-1)(2r^2u-2ru+2r-u-1)}{12}. \label{gapshere}
\end{align}
Therefore we obtain
$$| G(P_1, P_2)| =\sum_{n \in G(P_1)} n + \sum_{n \in G(P_2)} n - r(P_1,P_2) =\frac{ur(r-1)(3r^2u-5ru+4r+4u-2)}{12}.$$

By \cite[Theorem 4.3]{MQS2016},
$$\Gamma(P_\infty,P_1)=\left\{ \left(mr-mj-ri, m\left(j-1\right)+i\right) \mid 1 \leq i \leq m-1- u , 1\leq j \leq r-1-\left\lfloor \frac{ri}{m} \right\rfloor \right\}.$$
For $(i_0, j_0)\in\N^2$ with $1 \leq i_0 \leq m-1- u$ and $1\leq j_0 \leq r-1-\left\lfloor \frac{ri_0}{m} \right\rfloor$,
as above we need to count the number $s_{i_0,j_0}$ of pairs $(i_1,j_1)\in\N^2$ such that
\begin{small}
\begin{equation}\label{eq1}
1 \leq i_1 \leq m-1- u,\, 1\leq j_1 \leq r-1-\left\lfloor \frac{ri_1}{m} \right\rfloor,\, 
m\left(j_1-j_0\right) < r\left(i_0-i_1\right),\, m\left(j_1-j_0\right) < \left(i_0-i_1\right).
\end{equation}
\end{small}
For $h\in\{0,1\}$ write $i_h=k_h u + t_h$, with $k_h\in\{0,\ldots,r-2\}$ and $t_h\in\{1,\ldots,u\}$.
Then $\left \lfloor \frac{ri_h}{m} \right\rfloor = k_h$.
We split $s_{i_0,j_0}$ in a number of cases:
\begin{itemize}
\item $i_1=i_0$. Then \eqref{eq1} implies $1\leq j_1\leq j_0-1$. 
\item $i_1>i_0$, $k_1>k_0$, and $t_1\leq t_0$. Then \eqref{eq1} implies $k_1-k_0+1\leq j_0\leq r-1-k_0$ and $1\leq j_1\leq k_0-k_1+j_0$.
\item $i_1>i_0$, $k_1\geq k_0$, and $t_1> t_0$. Then \eqref{eq1} implies $k_1-k_0+2\leq j_0\leq r-1-k_0$ and $1\leq j_1\leq k_0-k_1-1+j_0$.
\item $i_1<i_0$ and $k_1<k_0$. Then \eqref{eq1} implies $1\leq j_1\leq j_0$.
\item $i_1<i_0$, $k_1=k_0$ and $t_1<t_0$. Then \eqref{eq1} implies $1\leq j_1\leq j_0$.
\end{itemize}
By direct computation, this yields
\begin{align*}
r(P_\infty, P_1)& =\sum_{(i_0,j_0)\in\Gamma(P_\infty,P_1)} s_{i_0,j_0}
= \sum_{k_0=0}^{r-2}\sum_{t_0=1}^{u} \sum_{j_0=1}^{r-1-k_0}(j_0-1)\\
&+\sum_{k_0=0}^{r-2}\sum_{t_0=1}^{u} \sum_{k_1=k_0+1}^{r-2} \sum_{t_1=1}^{t_0} \sum_{j_0=k_1-k_0+1}^{r-1-k_0}(k_0-k_1+j_0) \\
&+ 
\sum_{k_0=0}^{r-2}\sum_{t_0=1}^{u} \sum_{k_1=k_0}^{r-2} \sum_{t_1=t_0+1}^{u} \sum_{j_0=k_1-k_0+2}^{r-1-k_0}(k_0-k_1-1+j_0)\\
& + \sum_{k_0=0}^{r-2}\sum_{t_0=1}^{u} \sum_{k_1=0}^{k_0-1} \sum_{t_1=1}^{u} \sum_{j_0=1}^{r-1-k_0}j_0+
\sum_{k_0=0}^{r-2}\sum_{t_0=1}^{u}\sum_{t_1=1}^{t_0-1} \sum_{j_0=1}^{r-1-k_0}j_0
=\frac{u(r-1)r(ur^2+r-u-5)}{12}.
\end{align*}
Also, by \cite[Th. 3.2]{MQS2016}, we have
\begin{align*}
\sum_{n \in G(P_\infty)} n & =  \sum_{i=1}^{m-1-u} \sum_{j=1}^{r-1-\left\lfloor \frac{ri}{m} \right\rfloor} \left(mr-mj-ri\right)\\
& =
\sum_{k=0}^{r-2} \sum_{t=1}^{u}\sum_{j=1}^{r-1-k} \left(mr-mj-r(ku+t)\right) 
= \frac{ur(r-1)(2ur^2-ur+r-2)}{12},
\end{align*}
and $\sum_{n \in G(P_1)} n$ was computed in \ref{gapshere}. Therefore we obtain
$$| G(P_1, P_2)| =\sum_{n \in G(P_1)} n - \sum_{n \in G(P_2)} n + r(P_1,P_2) =\frac{ur(r-1)(3r^2u-5ru+4r+4u-2)}{12}.$$
\end{proof}

\begin{remark}
If $\mathcal H$ is the function field of the Hermitian curve defined by $y^{q+1}=x^q+x$ over $\mathbb{F}_{q^2}$, then Theorem {\rm \ref{twopoints}} was already obtained in {\rm \cite[Th. 3.6]{G2001}}. In fact, the places of $\mathcal H$ which are totally ramified in $ H/\mathbb F_{q^2}(x)$ are centered at Weierstrass points of $\mathcal H$.
\end{remark}

\section{Pure gaps at many points and codes}\label{Sec:ManyPoints}

Throughout this section, $F/\fq(x)$ is a Kummer extension defined by $y^m=f(x)$, where $f(x)\in\fq[x]$ is separable of degree $r$ and $\gcd(r,m)=1$. Let $P_\infty\in\bP(F)$ denote the unique pole of $x$, while $P_1,\ldots,P_s$ ($s\geq1$) are other totally ramified places in $F/\fq(x)$ different from $P_\infty$. 
In this section we give arithmetic conditions which characterize the pure gaps at $P_1,\ldots,P_s$ and at $P_\infty,P_1,\ldots,P_s$.
We use this characterization to determine explicit families of pure gaps at many points and apply it to construct AG codes with good parameters.

\begin{proposition}\label{puregapsmanypoints} 
Under the above notation, let $s\leq r$. The $s$-tuple $(a_1, \dots, a_s) \in\N^s$ is a pure gap at $P_1, \dots, P_s$ if and only if, for every $t \in \{0,\ldots ,m - 1\}$, exactly one of the following two conditions is satisfied:
\begin{enumerate}
\item[i)] $\sum_{i=1}^s \left\lfloor \frac{a_i+t}{m}\right\rfloor +\left\lfloor \frac{-rt}{m}\right\rfloor<0$;
\item[ii)] $\sum_{i=1}^s \left\lfloor \frac{a_i+t}{m}\right\rfloor +\left\lfloor \frac{-rt}{m}\right\rfloor \geq 0$ and $\left\lfloor \frac{a_i+t}{m}\right\rfloor=\left\lfloor \frac{a_i-1+t}{m}\right\rfloor$, for all $i=1, \dots ,s$.
\end{enumerate}
\end{proposition}
\begin{proof}
Let $P_1,\ldots,P_r$, be all the places of $F$ which are totally ramified in $F/\fq(x)$ except $P_\infty$, that is, $P_i$ is the zero of $x-\alpha_i$, where $f(x)=\prod_{i=1}^r (x-\alpha_i)$ is the separable polynomial defining $F$ by $y^m=f(x)$.
Then the divisor of $y$ in $F$ is $(y)=\sum_{i=1}^{r}P_i-rP_{\infty}$, and hence, for any $t\in\{0,\ldots,m-1\}$,
$$\sum_{i=1}^s a_iP_i+(y^t)=\sum_{i=1}^s (a_i+t)P_i+\sum_{i=s+1}^{r}tP_i- rtP_{\infty}\,.$$
Let $Q_1,\ldots,Q_r,Q_\infty$ be the places of $\fq(x)$ lying under $P_1,\ldots,P_r,P_\infty$, respectively. Then
$$
\left[\sum_{i=1}^s a_iP_i+\left(y^t\right) \right]\Big|_{{K(x)}} = \sum_{i=1}^s \left \lfloor \frac{a_i+t}{m}\right\rfloor Q_i +\left \lfloor \frac{-rt}{m}\right\rfloor Q_{\infty}.
$$
Since
$$
\mathcal{L}(\sum_{i=1}^s a_iP_i)=\bigoplus_{t=0}^{m-1}\mathcal{L}\left(\left[\sum_{i=1}^s a_iP_i+\left(y^t\right) \right]\Big |_{K(x)} \right)y^t,
$$
by Theorem \ref{ThMaharaj}, we have
$$
\ell\left(\sum_{i=1}^s a_iP_i\right)=\sum_{t=0}^{m-1}\ell\left( \sum_{i=1}^s \left \lfloor \frac{a_i+t}{m}\right\rfloor Q_i+\left \lfloor \frac{-rt}{m}\right\rfloor Q_{\infty}\right),
$$
$$
\ell\left(\sum_{i=1}^s (a_i-1)P_i\right)=\sum_{t=0}^{m-1} \ell\left(  \sum_{i=1}^s \left \lfloor \frac{a_i-1+t}{m}\right\rfloor Q_i+\left \lfloor \frac{-rt}{m}\right\rfloor Q_{\infty}\right).
$$
By Lemma \ref{purecondition}, $(a_1, \dots, a_s)$ is a pure gap at $P_1,\ldots,P_s$ if and only if 
$$
\ell\left( \sum_{i=1}^s \left \lfloor \frac{a_i+t}{m}\right\rfloor Q_i +\left \lfloor \frac{-rt}{m}\right\rfloor Q_{\infty}\right)-\ell\left( \sum_{i=1}^s \left \lfloor \frac{a_i-1+t}{m}\right\rfloor Q_i+\left \lfloor \frac{-rt}{m}\right\rfloor Q_{\infty}\right)=0
$$
for all $t\in \{0,\ldots,m-1\}$. Since $\fq(x)$ has genus $0$, this happens if and only if, for all $t\in \{0,\ldots,m-1\}$, either
$$\sum_{i=1}^s \left \lfloor \frac{a_i+t}{m}\right\rfloor+\left \lfloor \frac{-rt}{m}\right\rfloor<0$$
or
$$\sum_{i=1}^s \left \lfloor \frac{a_i+t}{m}\right\rfloor+\left \lfloor \frac{-rt}{m}\right\rfloor \geq 0 \quad{\rm and }\quad \sum_{i=1}^s \left \lfloor \frac{a_i+t}{m}\right\rfloor  = \sum_{i=1}^s \left \lfloor \frac{a_i-1+t}{m}\right\rfloor.$$
\end{proof}

\begin{proposition}\label{puregapsinfty}
Let $s\leq r$, then an $(s+1)$-tuple $(a_0, a_1, \dots, a_s)\in \mathbb{N}^{s+1}$ is a pure gap at $P_\infty, P_1, \dots, P_s$ if and only if, for every $t \in \{0,\ldots ,m - 1\}$, exactly one of the following two conditions is satisfied:
\begin{enumerate}
\item[i)] $\sum_{i=1}^s \left\lfloor \frac{a_i+t}{m} \right\rfloor + \left\lfloor \frac{a_0-rt}{m}\right \rfloor  <0 $;
\item[ii)] $\sum_{i=1}^s \left\lfloor \frac{a_i+t}{m} \right\rfloor + \left\lfloor \frac{a_0-rt}{m}\right \rfloor  \geq 0 $, $\left\lfloor \frac{a_0-rt}{m}\right \rfloor= \left\lfloor \frac{a_0-1-rt}{m}  \right\rfloor $ and $ \left\lfloor \frac{a_i+t}{m} \right\rfloor = \left\lfloor \frac{a_i-1+t}{m} \right\rfloor$ for $i=1, \dots, s$.
\end{enumerate}
\end{proposition}
\begin{proof}
The proof is very similar to the proof of Proposition \ref{puregapsmanypoints} and it is omitted.
\end{proof}

We now present three families of pure gaps at two points for $m\equiv1\pmod r$.
 
\begin{proposition}\label{puregapstwopoints}
Suppose that $m=ur+1$ for some integer $u$. Then
\begin{enumerate}
\item[i)] $((r-1)m-2r,1)$ is a pure gap at $P_{\infty},P_1$;
\item[ii)] $((r-2)m-r,b)$, with $b\in \{1,\ldots,u+1\}$ are pure gaps at $P_{\infty},P_1$;
\item[iii)] $((r-3)m+1+\alpha ,1+\beta)$, with $\alpha\in \{0,\ldots,2u-1\}$ and $\beta \in \{0,\ldots,u-1\}$ are pure gaps at $P_1,P_2$.
\end{enumerate} 
\end{proposition}
\begin{proof}
Let $a= rm-m-2r$ and $t\in \{0,\ldots,m-1\}$. We have
$\left\lfloor \frac{a-rt}{m}\right\rfloor \neq \left\lfloor \frac{a-1-rt}{m}\right\rfloor$ if and only if $m$ divides $a-rt=(r-1)m-r(t+2)$, that is $t=m-2$.
Also, $t=m-2$ implies $\left\lfloor \frac{a-qt}{q^{\ell}+1}\right\rfloor=-1$.
For any $t\in \{0,\ldots,m-2\}$ we have $\left\lfloor \frac{1+t}{m}\right\rfloor =\left\lfloor \frac{t}{m}\right\rfloor=0$.
We conclude that for any $t\in\{0,\ldots,m-2\}$ either 
$\left\lfloor \frac{a-rt}{m}\right\rfloor+\left\lfloor \frac{1+t}{m}\right\rfloor<0$ or $\left\lfloor \frac{a-rt}{m}\right\rfloor+\left\lfloor \frac{1+t}{m}\right\rfloor=\left\lfloor \frac{a-1-rt}{m}\right\rfloor+\left\lfloor \frac{t}{m}\right\rfloor.$
For $t=m-1$, $\left\lfloor \frac{a-rt}{m}\right\rfloor+\left\lfloor \frac{1+t}{m}\right\rfloor=-2+1=-1<0.$
By Proposition \ref{puregapsinfty}, $(a,1)$ is a pure gap at $P_\infty, P_1$. 

Now let $a=rm-2m-r$, $b\in\{1,\ldots,u+1\}$, and $t\in \{0,\ldots,m-1\}$. We have that $\left\lfloor \frac{b+t}{m}\right\rfloor\in\{0,1\}$, and $\left\lfloor \frac{b+t}{m}\right\rfloor=1$ if and only if 
$t+b\geq m$, that is $t \in \{m-b,\dots, m-1\}$. 
In this case,
$$\left\lfloor \frac{a-rt}{m}\right\rfloor = \left\lfloor \frac{rm-2m-r-rt}{m}\right\rfloor =-2 +\left\lfloor \frac{rm-r-rt}{m}\right\rfloor =-2,$$
since $0\leq rm-r-rt\leq r(b-1) \leq ru < m $.
Hence, for all $t \in \{m-b,\dots, m-1\}$,
$$\left\lfloor \frac{a-rt}{m}\right\rfloor+\left\lfloor \frac{b+t}{m}\right\rfloor=-2+1<0.$$
For $t \in \{0,\dots, m-b-1\}$, we have that
$$\left\lfloor \frac{a-rt}{m}\right\rfloor+\left\lfloor \frac{b+t}{m}\right\rfloor=
\left\lfloor \frac{a-rt}{m}\right\rfloor=
\left\lfloor \frac{a-1-rt}{m}\right\rfloor=\left\lfloor \frac{a-1-rt}{m}\right\rfloor+\left\lfloor \frac{b-1+t}{m}\right\rfloor.$$
By Proposition \ref{puregapsinfty}, $(a,b)$ is a pure gap at $P_\infty, P_1$.

Finally let $t \in \{ 0, \dots, m-1 \}$ and $(a_\alpha, b_\beta)=((r-3)m+1+\alpha,1+\beta)$ with $ \alpha\in \{0,\ldots,2u-1\}$ and $\beta\in \{0,\ldots,u-1\}$.
Note that $\left\lfloor \frac{a_{\alpha}+t}{m}\right\rfloor \neq \left\lfloor \frac{a_{\alpha}-1+t}{m}\right\rfloor$ if and only if 
$t=m-1-\alpha$, and $\left\lfloor \frac{b_{\alpha}+t}{m}\right\rfloor \neq \left\lfloor \frac{b_{\alpha}-1+t}{m}\right\rfloor$ if and only if $t=m-1-\beta$.
Therefore,
$$\left\lfloor \frac{a_{\alpha}+t}{m}\right\rfloor +\left\lfloor \frac{b_{\alpha}+t}{m}\right\rfloor \neq \left\lfloor \frac{a_{\alpha}-1+t}{m}\right\rfloor+\left\lfloor \frac{b_{\alpha}-1+t}{m}\right\rfloor$$
if and only if $t=m-1-\alpha$ or $t=m-1-\beta$.

Suppose $t=m-1-\alpha$. Then
\begin{align*}
& \left\lfloor \frac{-rt}{m}\right\rfloor=-r+\left\lfloor \frac{r(1+\alpha)}{m}\right\rfloor
= \left\{\begin{array}{ll} -r,& \alpha \leq u-1 \\ -r +1,& \alpha\geq u\\ \end{array} \right. ,\\
&\left\lfloor \frac{a_{\alpha}+t}{m}\right\rfloor=r-2, \quad \left\lfloor \frac{b_{\beta}+t}{m}\right\rfloor=
1+ \left\lfloor \frac{\beta - \alpha}{m}\right\rfloor = \left\{ \begin{array}{ll} 1,& \text{ for }\beta\geq \alpha\\ 0, & \text{ for } \beta<\alpha\\ \end{array} \right..
\end{align*}
If $\alpha\geq u$, then
$$\left\lfloor \frac{-rt}{m}\right\rfloor+\left\lfloor \frac{a_{\alpha}+t}{m}\right\rfloor+\left\lfloor \frac{b_{\beta}+t}{m}\right\rfloor=(-r+1)+(r-2)+0<0;$$
if $\alpha\leq u-1$, then
$$\left\lfloor \frac{-rt}{m}\right\rfloor+\left\lfloor \frac{a_{\alpha}+t}{m}\right\rfloor+\left\lfloor \frac{b_{\beta}+t}{m}\right\rfloor\leq-r+(r-2)+1<0.$$

Suppose $t=m-1-\beta$. Then
\begin{align*}
&\left\lfloor \frac{-rt}{m}\right\rfloor=
-r+\left\lfloor \frac{r(1+\beta)}{m}\right\rfloor=-r,\\
& \left\lfloor \frac{a_{\alpha}+t}{m}\right\rfloor=
r-2+\left\lfloor \frac{\alpha-\beta}{m}\right\rfloor=
\left\{ \begin{array}{ll} r-3, & \text{ for } \alpha<\beta\\ r-2, & \text{ for } \alpha\geq \beta\\\end{array}\right., \qquad \left\lfloor \frac{b_{\beta}+t}{m}\right\rfloor=1.
\end{align*}

Hence,
\begin{align*}
\left\lfloor \frac{-rt}{m}\right\rfloor+\left\lfloor \frac{a_{\alpha}+t}{m}\right\rfloor+\left\lfloor \frac{b_{\beta}+t}{m}\right\rfloor\leq -r +(r-2)+1<0.
\end{align*}
The thesis follows from Proposition \ref{puregapsmanypoints}.
\end{proof}

The following results present two families of pure gaps at many points for $m\equiv1\pmod r$.

\begin{proposition}\label{manypoints1}
Suppose that $m=ur+1$ for some integer $u$, $s<r$, and $\alpha_i\in\{0,\ldots,(s+1-i)u-1\}$ for $i=1,\ldots,s$.
Then $(a_1,\ldots,a_s)=((r-s-1)m+1+\alpha_1,1+\alpha_2,\ldots,1+\alpha_s)$ is a pure gap at $P_1,\ldots,P_s$.
\end{proposition}
\begin{proof}
Suppose there exist $t\in\{0,\ldots,m-1\}$ and $j\in\{1,\ldots,s\}$ such that $\left\lfloor\frac{a_j+t}{m}\right\rfloor\ne\left\lfloor\frac{a_j-1+t}{m}\right\rfloor$. Thus $t=m-1-\alpha_j$.
Let $h\in\{0,\ldots,r-2\}$ be such that $hu\leq\alpha_j< (h+1)u$. We have
\begin{align*}
&\left\lfloor \frac{-rt}{m}\right\rfloor=\left\lfloor \frac{-r(m-1-\alpha_j)}{m}\right\rfloor=-r+\left\lfloor \frac{r(1+\alpha_j)}{m}\right\rfloor
= -r+h,\\
&\left\lfloor \frac{a_{1}+t}{m}\right\rfloor=\left\lfloor \frac{(r-s-1)m+1+\alpha_1+m-1-\alpha_j}{m}\right\rfloor=\left\{ \begin{array}{ll} r-s,& \alpha_1\geq \alpha_j\\ r-s-1, & \alpha_1< \alpha_j\\ \end{array} \right.,
\end{align*}
and, for $i>1$,
$$\left\lfloor \frac{a_{i}+t}{m}\right\rfloor=\left\lfloor \frac{1+\alpha_i+m-1-\alpha_j}{m}\right\rfloor=\left\{ \begin{array}{ll} 0,&  \alpha_i< \alpha_j\\ 1, & \alpha_i\geq \alpha_j\\ \end{array} \right..$$
Since
$$ |\{i\in\{2,\ldots,s\}:\alpha_i\geq\alpha_j\}|\leq s-1-|\{i\in\{2,\ldots,s\}: (s+1-i)h-1<uh\}|=s-1-h, $$
this implies that
$$\left\lfloor \frac{-rt}{m}\right\rfloor+\left\lfloor \frac{a_{1}+t}{m}\right\rfloor+\sum_{i=2}^m\left\lfloor \frac{a_{i}+t}{m}\right\rfloor \leq (-r+h)+(r-s)+(s-1-h)<0. $$
Hence, the thesis follows by Proposition \ref{puregapsmanypoints}.
\end{proof}

\begin{proposition}\label{manypoints2}
Suppose that $m=ur+1$ for some integer $u$, $s<r-1$, $\alpha\in \{0,\ldots,s\}$, and $\beta_i\in \{0,\ldots,iu-1\}$ for $i\in\{1,\ldots,s\}$.
Then $(a_0,a_1,\ldots,a_s)=((r-s-1)m-r+\alpha,1+\beta_1,\ldots,1+\beta_s)$ is a pure gap at $P_{\infty},P_1,\ldots,P_s$.
\end{proposition}
\begin{proof}

Let $t\in\{0,\ldots,m-2\}$, so that $t=ku+z$ with $k\in\{0,\ldots,r-1\}$ and $z\in\{0,\ldots,u-1\}$.

Suppose $\left\lfloor \frac{a_0-rt}{m} \right\rfloor \ne \left\lfloor \frac{a_0-1-rt}{m} \right\rfloor$.
Then $m\mid(a_0-rt)=(r-s-k-1)m+\alpha+k-r(z+1)$. Since $|\alpha+k-r(z+1)|<m$, this implies $\alpha+k=r(z+1)$, whence $r\mid(\alpha+k)$. As $0\leq\alpha,k\leq r-1$, and $r(z+1)>0$, we have that $\alpha+k=r$ and $z=0$. Hence, $t=m-1-\alpha u$.
Then
$$ \left\lfloor\frac{a_0-rt}{m}\right\rfloor=r-s-k-1=\alpha-s-1 . $$
Also, $1+\beta_i+t\leq m-1-(\alpha-j)u$ for all $i$.
Thus $a_j+t\leq m-1$ for all $j\in\{1,\ldots,\alpha\}$, so 
$$ \sum_{i=1}^s \left\lfloor\frac{a_i+t}{m}\right\rfloor \leq s-\alpha. $$
Therefore,
$$ \sum_{i=1}^s \left\lfloor\frac{a_i+t}{m}\right\rfloor + \left\lfloor\frac{a_0-rt}{m}\right\rfloor <0. $$

Now suppose $\left\lfloor\frac{a_i+t}{m}\right\rfloor\ne\left\lfloor\frac{a_j-1+t}{m}\right\rfloor$ for some $j\in\{1,\ldots,s\}$. Since $1\leq a_j+t<2m$, this implies $t=m-a_j=m-1-\beta_j$.
Let $h\in\{0,r-3\}$ be such that $hu\leq\beta_j<(h+1)u$. We have
$$ \left\lfloor\frac{a_0-rt}{m}\right\rfloor=-s-1+\left\lfloor\frac{\alpha+r\beta_j}{m}\right\rfloor=-s-1+h $$
and, for $i>0$,
$$ \left\lfloor\frac{a_i+t}{m}\right\rfloor=1+\left\lfloor\frac{\beta_i-\beta_j}{m}\right\rfloor=\left\{ \begin{array}{ll} 0,&  \beta_i< \beta_j\\ 1, & \beta_i\geq \beta_j\\ \end{array} \right.. $$
Since
$$ |\{i\in\{1,\ldots,s\}:\beta_i\geq\beta_j\}|\leq s-|\{i\in\{1,\ldots,s\}: ih-1<uh\}|=s-h, $$
this implies that
$$\left\lfloor \frac{a_0-rt}{m}\right\rfloor+\sum_{i=1}^m\left\lfloor \frac{a_{i}+t}{m}\right\rfloor \leq (-s-1+h)+(s-h)<0. $$

Finally, let $t=m-1$. Then $\left\lfloor\frac{a_0-rt}{m}\right\rfloor=-s-1$ and $\left\lfloor\frac{a_i+t}{m}\right\rfloor=1$ for all $i>0$. Hence,
$$\left\lfloor \frac{a_0-rt}{m}\right\rfloor+\sum_{i=1}^m\left\lfloor \frac{a_{i}+t}{m}\right\rfloor = (-s-1)+s<0. $$
The thesis follows by Proposition \ref{puregapsinfty}.
\end{proof}

By means of Theorem \ref{distmanypoints}, the results on pure gaps of this section can be used in order to obtain AG codes with good parameters. This is pointed out in the next remarks where we compute an upper bound for the Singleton defect of some codes.

\begin{remark}\label{applications1}
For a Kummer extension $y^m=f(x)$, where $m=ur+1$ and $s\leq r-1$, consider the pure gaps $(a_1,\ldots,a_s)=((r-s-1)m+1,1,\ldots,1)$ and $(b_1,\ldots,b_s)=((r-s-1)m+su,(s-1)u,\ldots,u)$.
Define the divisors $G=\sum_{i=1}^s (a_i+b_i-1)P_i$ and $D$ as the sum of $n$ rational places of $F$ different from $P_1,\ldots,P_s$. Consider the $[n,k,d]$-code $C_{\Omega}(D,G)$.

Suppose $2g-2 < \deg G < n$, then  $k=n+g-1-\deg G$. Since $F$ has genus $g=ur(r-1)/2$ we have by Proposition {\rm\ref{manypoints1}} and Theorem {\rm\ref{distmanypoints}} that 
the Singleton defect $\delta=n+1-k-d$ satisfies
$$ \delta\leq \frac{ur(r-1)-us(s+1)}{2}. $$
\end{remark}

\begin{remark}\label{applications2}
For a Kummer extension $y^m=f(x)$, where $m=ur+1$ and $s\leq r-2$ consider the pure gaps $(a_0,a_1,\ldots,a_s)=((r-s-1)m-r,1,\ldots,1)$ and $(b_0,b_1,\ldots,b_s)=((r-s-1)m-r+s,u,\ldots,su)$.
Define the divisors $G=(a_0+b_0-1)P_\infty+\sum_{i=1}^s (a_i+b_i-1)P_i$ and $D$ as the sum of $n$ rational places of $F$ different from $P_\infty,P_1,\ldots,P_s$ and consider the $[n,k,d]$-code $C_{\Omega}(D,G)$.

Suppose $2g-2 < \deg G < n$, then  $k=n+g-1-\deg G$. Since $F$ has genus $g=ur(r-1)/2$ we have by Proposition {\rm \ref{manypoints2}} and Theorem {\rm\ref{distmanypoints}} that 
the Singleton defect $\delta$ satisfies
$$ \delta\leq \frac{ur(r-1)-us(s+1)}{2}-s-1.$$
\end{remark}

We illustrate the results obtained by constructing codes on many points over the Hermitian function field. 

\begin{example}\label{ExHerm}
The Hermitian function field $\mathcal H$ is defined by the affine equation $y^{q+1}=x^q+x$, it is maximal over $\mathbb{F}_{q^2}$ and has genus $g=q(q-1)/2$. We apply Remark \ref{applications1} to construct $[n,k,d]$-codes $C_{\Omega}(D,G)$ from $\mathcal H$. In this case we have $r=q,  u=1, 1 \leq s \leq q-1$ and $ \deg G= 2(q-s-1)(q+1)+s(s+1)/2$. 
We choose $s$ such that $2g-2 < \deg G < n$ with $n=q^3+1-s$. Then
\begin{align*}
& k=n+g-1-\deg G=q^3-\frac{3}{2}q^2+\Big(2s-\frac{1}{2}\Big)q-\frac{s^2-s}{2}+2,\\
& d \geq \deg G -(2g-2)+s+\sum_{i=1}^s (b_i-a_i)=q^2-(2s-1)q+s^2-s.
\end{align*}
\end{example}
Table \ref{tabellina} summarizes results from Example \ref{ExHerm}. We list AG codes with the same or better parameters with respect to the corresponding ones in the MinT's Tables \cite{MinT}.

\begin{center}
\begin{table}
\caption{Results from Example \ref{ExHerm}}\label{tabellina}
\vspace*{0.3 cm}
\begin{center}
\begin{tabular}{|c|c|c|c|c|c|}
\hline 
$q^2$ & $s$ & $n$ & $k$ & $d\geq$ & \text{improvement on $d$ compared with ~\cite{MinT}} \\ 
\hline 
16 & 1 & 64 & 48 & 12 & 1\\ 
\hline 
16 & 2 & 63 & 55 & 6 & 0 \\ 
\hline 
25 & 1 & 125 & 97 & 20 & 1 \\ 
\hline 
25 & 2 & 124 & 106 & 12 & 1 \\ 
\hline 
49& 2 & 342 & 295 & 30 & 3 \\ 
\hline 
49 & 3 & 341 & 307 & 20 & 1 \\ 
\hline 
64 & 1 & 512 & 430 & 56 & 1\\ 
\hline 
64 & 2 & 511 & 445 & 42 & 3\\ 
\hline
64 & 3 & 510 & 459 & 30 & 2\\ 
\hline 
64 & 4 & 509 & 472 & 20 & 0 \\ 
\hline 
81 & 3 & 727 & 656 & 42 & 3\\ 
\hline 
81 & 4 & 726 & 671 & 30 & 0 \\ 
\hline 
\end{tabular} 
\end{center}
\end{table}
\end{center}

\section{Acknowledgments}
The research of D. Bartoli and G. Zini 
was partially supported by Ministry for Education, University
and Research of Italy (MIUR) (Project PRIN 2012 ``Geometrie di Galois e
strutture di incidenza'' - Prot. N. 2012XZE22K$_-$005)
 and by the Italian National Group for Algebraic and Geometric Structures
and their Applications (GNSAGA - INdAM).
The second  author L. Quoos was partially supported by CNPq, PDE grant number 200434/2015-2. This work was done while the author enjoyed a sabbatical at the Universit\`a degli Studi di Perugia leave from Universidade Federal do Rio de Janeiro.

\end{document}